\documentclass{amsart}

\usepackage{}
\usepackage{tikz}
\usepackage{tikz-cd}
\newtheorem{theorem}{Theorem}
\newtheorem{lemma}{Lemma}
\newtheorem{corollary}{Corollary}
\newtheorem{prop} {Proposition}

\theoremstyle{definition}
\newtheorem{definition}{Definition}

\theoremstyle{remark}
\newtheorem{remark}{Remark}

\newcommand{\C}{\mathbb C}
\newcommand{\hatC}{\widehat{\C}}
\numberwithin{equation}{section}
\newcommand{\T}{Teich\-m\"ul\-ler}

\usepackage{tikz}
\usetikzlibrary{matrix,arrows}
\usepackage[all]{xy}
\usepackage{pdfpages}

\makeatletter
\@namedef{subjclassname@2020}{%
  \textup{2020} Mathematics Subject Classification}
\makeatother

\begin{document}

\title[Teichm\"uller space]{Teichm\"uller space of a closed set in the Riemann sphere}

\author{Xinlong Dong}
\address[Dong] {Department of Mathematics and Computer Science\\
	Kingsborough Community College of the City University of New York\\
	Brooklyn, NY 11235--2398, USA}

\email{Xinlong.Dong@kbcc.cuny.edu}

 \author{Arshiya Farhath.~G}
\address[Farhath]{Department of Mathematics\\
 Graduate Center of the City University of New York\\
 New York, NY 10016, USA}

\email[Farhath]{agulamdasthagir@gradcenter.cuny.edu}

\author{Sudeb Mitra}
\address[Mitra]{Department of Mathematics,
Queens College of the City University of New York\\
Flushing, NY 11367-1597, USA\\
 and\\
 Department of Mathematics\\
 Graduate Center of the City University of New York
 New York, NY 10016, USA}

\email[Mitra]{sudeb.mitra@qc.cuny.edu}

\dedicatory{For Professor Yunping Jiang on his 65th Birthday}

\thanks{The first and third authors are supported by PSC-CUNY research grants.}


\begin{abstract}
The {\T} space of a closed set in the Riemann sphere is a simply connected complex Banach manifold. Its complex structure follows from {\it Lieb isomorphism}. In this paper, we show the conformal naturality of Lieb isomorphism. We then study Douady-Earle section for these {\T} spaces. In particular, we study the real-analyticity of Douady-Earle section for classical {\T} spaces. We give two explicit examples of maximal holomorphic motions over simply connected complex Banach manifolds. As an application of the real-analyticity of the Douady-Earle section for the classical {\T} spaces of Riemann surfaces, we prove a new result showing that a family of Jordan curves varies real-analytically over a simply connected complex Banach manifold and as quasiconformal images of the one at the basepoint, provided that a finite number of marked points on the Jordan curves vary holomorphically over the same parameter space. 

\end{abstract}

\maketitle

\section{Introduction}

{\bf Notations.} We will use the following notation: $\C$ for the complex plane, $\hatC: = \C \cup\{\infty\}$ for the Riemann sphere, and $\Delta$ for the open unit disk $\Delta:= \{z \in \C: \vert z\vert <1\}$. 

\medskip

Associated to each closed set $E$ in $\hatC$, there is a {\T} space $T(E)$, which is a universal parameter space for holomorphic motions of the set $E$. This space was first studied by G.~S.~Lieb in his 1990 Cornell University doctoral dissertation; see \cite{Li}. We give the precise definitions in \S\S3.1 and 3.2. The complex analytic structure of $T(E)$ follows from {\it Lieb isomorphism}; see \S3.4 for the precise definition. Theorem A of our paper shows the conformal naturality of Lieb isomorphism; see \S\S3.5, 3.6, and 3.7. In \S4, we discuss {\it universal holomorphic motion} of the closed set $E$. In \S5, we discuss {\it Douady-Earle section} for classical {\T} spaces, for product {\T} spaces, and for the {\T} space $T(E)$. In particular, we use some facts from \cite{DE} to show the real-analyticity of the Douady-Earle section for classical {\T} spaces. This important fact was not explicitly stated in the paper \cite{DE}. We also include an open question. Finally, we study two applications of the real-analyticity of the Douady-Earle section for some finite-dimensional {\T} spaces. In \S6, we study an explicit example of a {\it maximal holomorphic motion} of a finite set $E$ over its {\T} space, which is Theorem B of our paper (see \S4 for the definition of a maximal holomorphic motion). In \S7, we review several results on holomorphic families of Jordan curves. We then state and prove Theorem C, which is an application of the real-analyticity of the Douady-Earle section for the classical {\T} spaces of Riemann surfaces. Theorem C asserts that a family of Jordan curves varies real-analytically over a simply connected complex Banach manifold and as quasiconformal images of the one at the basepoint, provided that a finite number of marked points on each Jordan curve depend holomorphically on the same parameter space. 

\medskip

{\bf Acknowledgement.} We want to thank the referee for his/her meticulous reading of the paper, and for several extremely important suggestions that greatly helped us to improve our paper. We also want to thank Professor Yunping Jiang for some illuminating and helpful discussions that helped us to clarify several subtle points in this paper.

\section{Quasiconformal mappings}

\begin{definition}A complex-valued function $w=f(z)$ defined on a region $\Omega$ in $\C$ is called a quasiconformal mapping if it is a sense--preserving homeomorphism of $\Omega$ onto its image and its complex distributional derivatives
$$w_z = \frac{1}{2}\Big(\frac{\partial f}{\partial x} - i \frac{\partial f}{\partial y}\Big) \mbox { and } w_{\overline{z}} = \frac{1}{2}\Big(\frac{\partial f}{\partial x} + i \frac{\partial f}{\partial y}\Big)$$
are Lebesgue measurable locally square integrable functions on $\Omega$ that satisfy the inequality $\vert w_{\overline{z}}\vert \leq k \vert w_z\vert$ almost everywhere in $\Omega$, for some real number $k$ with $0\leq k < 1$. 
\end{definition}

\medskip

If $w = f(z)$ is a quasiconformal mapping defined on the region $\Omega$ then the function $w_z$ is known to be nonzero almost everywhere on $\Omega$. Therefore the function
$$\mu_f = \frac{w_{\overline{z}}}{w_z}$$
is a well--defined $L^{\infty}$ function on $\Omega$, called the {\it complex dilatation} or the {\it Beltrami coefficient} of $f$. The $L^{\infty}$ norm of every Beltrami coefficient is less than one. 

\medskip

The positive number 
$$K(f) = \frac{1 + \Vert\mu_f\Vert_{\infty}}{1 - \Vert\mu_f\Vert_{\infty}}$$
is called the {\it dilatation} of $f$. We say that $f$ is $K$-quasiconformal if $f$ is a quasiconformal mapping and $K(f) \leq K$.

\medskip

We call a homeomorphism of $\hatC$ {\it normalized} if it fixes the points $0$, $1$, and $\infty$. 

\medskip

Let $M(\C)$ denote the open unit ball of the complex Banach space $L^{\infty}(\C)$. The basepoint of $M(\C)$ is the zero function.
For each $\mu$ in $M(\C)$, there exists a unique normalized quasiconformal homeomorphism of $\hatC$ onto itself that has Beltrami coefficient $\mu$; we denote this quasiconformal map by $w^{\mu}$. Furthermore, for every fixed $z \in \hatC$, the map $\mu \mapsto w^{\mu}(z)$ of $M(\C)$ into $\hatC$ is holomorphic; see \cite{Ah} and \cite{AB}.

\section{Conformal naturality of Lieb isomorphism}

\subsection{Teichm\"uller space of a closed set in $\hatC$}

The Teichm\"uller space of a closed set $E$ in the Riemann sphere $\hatC$, denoted by $T(E)$, has important applications in the study of holomorphic motions of the set $E$. In this section, we study the conformal naturality of {\it Lieb isomorphism}; see {\bf Theorem A}.

 Let $E$ be a closed set in the Riemann sphere $\hatC$, such that $0$, $1$, and $\infty$ belong to $E$.

\begin{definition}

Two normalized quasiconformal self-mappings $f$ and $g$ of the Riemann sphere $\hatC$
are said to be $E$-equivalent if and only if $f^{-1} \circ g$ is
isotopic to the identity rel $E$. The {\it Teichm\"uller space}
$T(E)$ is the set of all $E$-equivalence classes of normalized
quasiconformal self-mappings of $\hatC$. 

\smallskip

The reader should note that by ``$f^{-1} \circ g$ is isotopic to the identity rel $E$" we mean that, at each stage of the isotopy, $f^{-1} \circ g$ keeps the set $E$ pointwise fixed.

\end{definition}

The basepoint of $T(E)$ is the $E$-equivalence class of the identity
map. 

\subsection{$T(E)$ as a complex manifold}

Let $M(\C)$ be the open unit ball of the complex Banach space
$L^{\infty}(\C)$. Each $\mu$ in $M(\C)$ is the Beltrami coefficient
of a unique normalized quasiconformal map $w^{\mu}$ of
$\hatC$ onto itself. The basepoint of $M(\C)$ is the zero function.

\smallskip

We define the quotient map 
$P_E: M(\C) \to T(E)$ by setting $P_E(\mu)$ equal to the
$E$-equivalence class of $w^{\mu}$, written as $[w^{\mu}]_E$.
Clearly, $P_E$ maps the basepoint of $M(\C)$ to the basepoint of
$T(E)$. 

\smallskip

In his doctoral dissertation (\cite{Li}), G.~Lieb proved that $T(E)$
is a complex Banach manifold such that the projection map $P_E: M(\C)
\to T(E)$ is a holomorphic split submersion. For more details, see \S3.4.

\subsection{Teichm\"uller metric on $T(E)$}

The Teichm\"uller distance $d_M(\mu, \nu)$ between $\mu$ and $\nu$
on $M(\C)$ is defined by
$$d_M(\mu, \nu) = \tanh^{-1}\Big\Vert\frac{\mu - \nu}{1 -
\bar\mu\nu}\Big\Vert_{\infty}.$$ The {\it Teichm\"uller metric} on
$T(E)$ is the quotient metric
$$d_{T(E)}(s,t) = \inf\{d_M(\mu,\nu): \mu \mbox { and } \nu \in M(\C), P_E(\mu)
= s \mbox { and } P_E(\nu) = t\}$$ for all $s$ and $t$ in $T(E)$.

\medskip

The Teichm\"uller metric on $T(E)$ is the same as its Kobayashi
metric. (See Proposition 7.30 in \cite{EM}.)

\subsection{Lieb isomorphism} We shall assume that $E$ is a closed set in $\hatC$, and has a nonempty complement $E^c = \hatC \setminus E$. Let $\{S_n\}$ be the connected components of $E^c$. Each $S_n$ is a hyperbolic Riemann surface; let $Teich(S_n)$ denote its Teichm\"uller space.  If the number of components is finite, $Teich(E^c)$ is, by definition, the cartesian product of the spaces $Teich(S_n)$. If there are infinitely many components, we define the {\it product Teichm\"uller space} $Teich(E^c)$ as follows.

\smallskip

For each $n \geq 1$, let $0_n$ be the basepoint of the Teichm\"uller space $Teich(S_n)$, and let $d_n$ be the Teichm\"uller metric on $Teich(S_n)$. Let $M(S_n)$ denote the open unit ball of the complex Banach space $L^{\infty}(S_n)$, for each $n \geq 1$. By definition, the {\it product Teichm\"uller space} $Teich(E^c)$ is the set of sequences $t = \{t_n\}_{n=1}^{\infty}$ such that $t_n$ belongs to $Teich(S_n)$ for each $n \geq 1$, and 
\begin{equation}
\label{eq:pts}
\sup\{d_n(0_n,t_n): n \geq 1\} < \infty.
\end{equation}
The basepoint of $Teich(E^c)$ is the sequence $0 = \{0_n\}$ whose $n$th term is the basepoint of $Teich(S_n)$. 

\smallskip

Let $L^{\infty}(E^c)$ be the complex Banach space of sequences $\mu = \{\mu_n\}$ such that $\mu_n$ belongs to $L^{\infty}(S_n)$ for each $n \geq 1$ and the norm $\Vert\mu\Vert_{\infty} = \sup\{\Vert\mu_n\Vert_{\infty}: n \geq 1\}$ is finite.  Let $M(E^c)$ be the open unit ball of $L^{\infty}(E^c)$. Note that if $\mu$ belongs to $M(E^c)$, then $\mu_n$ belongs to $M(S_n)$ for all $n \geq 1$ (but the converse is false). 

\smallskip

For each $n \geq 1$, let $\Phi_n$ be the standard projection from $M(S_n)$ to $Teich(S_n)$; see \cite{Hubbard} or \cite{Nag} for standard facts on classical {\T} spaces. For $\mu$ in $M(E^c)$, let $\Phi(\mu)$ be the sequence $\{\Phi_n(\mu_n)\}$. It is easy to see that $\Phi(\mu)$ belongs to $Teich(E^c)$, and the map $\Phi$ is surjective. We call $\Phi$ the {\it standard projection} of $M(E^c)$ onto $Teich(E^c)$.  In \cite{Li} it was shown that $Teich(E^c)$ is a complex Banach manifold such that the map $\Phi$ is a holomorphic split submersion (see also \cite{EM}). 

\smallskip

Let $M(E)$ be the open unit ball in $L^{\infty}(E)$. The product $Teich(E^c) \times M(E)$ is a complex Banach manifold. (If $E$ has zero area, then $M(E)$ contains only one point, and $Teich(E^c) \times M(E)$ is then isomorphic to $Teich(E^c)$.)

For $\mu$ in $L^{\infty}(\C)$, let $\mu|E^c$ and $\mu|E$ be the restrictions of $\mu$ to $E^c$ and $E$ respectively. We define the projection map $\widetilde P_E$ from $M(\C)$ to $Teich(E^c) \times M(E)$ by the formula:
$$\widetilde P_E(\mu) = (\Phi(\mu|E^c), \mu|E) \qquad \mbox { for all } \mu \in M(\C).$$

\begin{prop} (Lieb's isomorphism theorem.) For all $\mu$ and $\nu$ in $M(\C)$ we have $P_E(\mu) = P_E(\nu)$ if and only if $\widetilde P_E(\mu) = \widetilde P_E(\nu)$. 
\end{prop}

See Section 7.9 of \cite{EM} for a complete proof. 

\begin{definition} It follows from Proposition 1, that there is a well-defined bijection $\mathcal{L}: T(E) \to Teich(E^c) \times M(E)$ such that $\mathcal{L} \circ P_E = \widetilde P_E$, and $T(E)$ has a unique complex manifold structure such that $P_E$ is a holomorphic split submersion. We call the biholomorphic map $\mathcal{L}$ the {\it Lieb isomorphism}. 
\end{definition}

\subsection{Statement of Theorem A}

Let $E$ and $\widetilde{E}$ be two closed sets in $\hatC$ such that $0$, $1$, and $\infty$ belong to both $E$ and $\widetilde{E}$. Let $g$ be a M\"obius transformation such that $g(E) = \widetilde{E}$. Let $S = \hatC \setminus E$ and $\widetilde{S} = \hatC \setminus \widetilde{E}$. So we have $g(S) = \widetilde{S}$. 

\smallskip

Let $\{S_i\}$ be the connected components of $S$ and $\{\widetilde{S_i}\}$ be the connected components of $\widetilde{S}$. The map $g$ induces a biholomorphic map 
$$F_g: T(E) \to T(\widetilde{E})$$ as follows:
$$F_g([w^{\mu}]_E) = [\widehat{g} \circ w^{\mu} \circ g^{-1}]_{\widetilde{E}} \qquad \mbox { for } \mu \in M(\C)$$
where $\widehat{g}$ is the unique M\"obius transformation such that $\widehat{g} \circ w^{\mu} \circ g^{-1}$ fixes $0$, $1$, and $\infty$. 

\smallskip

Following \S2.2 in \cite{EGL}, there exists a bijection $\alpha: \mathbb{N} \to \mathbb{N}$ such that $g(S_i) = \widetilde{S}_{\alpha(i)}$ for every $i \in \mathbb{N}$. The M\"obius transformation $g$ induces a biholomorphic map 
$$(\rho_g)_i: Teich(S_i) \to Teich (\widetilde{S}_{\alpha(i)})$$ such that the following diagram commutes:
$$
\begin{array}{ccc}
M(S_i) & {\buildrel (\sigma_g)_i  \over
\longrightarrow} & M(\widetilde{S}_{\alpha(i)})\cr
\downarrow \Phi_i& &\downarrow \widetilde{\Phi}_{\alpha(i)}\cr Teich(S_i) & {\buildrel (\rho_g)_i \over \longrightarrow} &
Teich(\widetilde{S}_{\alpha(i)})
\end{array}
$$
Here, $(\rho_g)_i([w]) = [w \circ g^{-1}]$, where $[w]$ is the {\T} equivalence class of the quasiconformal map $w$ with domain $S_i$ and $[w \circ g^{-1}]$ is the {\T} equivalence class of the quasiconformal map $w \circ g^{-1}$ with domain $\widetilde{S}_{\alpha(i)}$. (And the map $(\sigma_g)_i$ sends the Beltrami coefficient of $w$ to the Beltrami coefficient of $w \circ g^{-1}$.)

\smallskip

The biholomorphic map $\rho_g: Teich(S) \to Teich(\widetilde{S})$ induced by $g$ is then defined by the formula:
$$\rho_g(x) = \{(\rho_g)_i(x_i)\} \qquad \mbox { for all } x = \{x_i\} \in Teich(S) \mbox { and } i \in \mathbb{N}.$$

We also have the standard map $g_{*}: M(E) \to M(\widetilde{E})$ given by 
$$g_{*}(\mu) = \nu \qquad \mbox { where } \mu = (\nu \circ g)\frac{\bar{g'}}{g'}.$$
See Equation (6.5) in \cite{DE}.

\smallskip

Following our discussion in \S3.4, we have the following {\it Lieb maps}: 
$$\mathcal{L}: T(E) \to Teich(E^c) \times M(E)$$ and 
$$\widetilde{\mathcal{L}}:T(\widetilde{E}) \to
Teich(\widetilde{E}^c) \times M(\widetilde{E}).$$

We can now state the first main theorem of this paper. 

\bigskip

{\bf Theorem A.} {\it The following diagram commutes:}
$$
\begin{array}{ccc}
T(E) & {\buildrel \mathcal{L} \over
\longrightarrow} & Teich(E^c) \times M(E)\cr
\downarrow F_g & &\downarrow (\rho_g \times g_{*})\cr T(\widetilde{E}) & {\buildrel \widetilde{\mathcal{L}} \over \longrightarrow} &
Teich(\widetilde{E}^c) \times M(\widetilde{E})
\end{array}
$$

\subsection{Proof of Theorem A}

For $\mu \in M(\C)$, let $[w^{\mu}] \in T(E)$. We have the biholomorphic map $F_g: T(E) \to T(\widetilde{E})$ given by: 
$$F_g([w^{\mu}]_E) = [\widehat{g} \circ w^{\mu} \circ g^{-1}]_{\widetilde{E}}$$ where $\widehat{g}$ is the unique M\"obius transformation such that $\widehat{g} \circ w^{\mu} \circ g^{-1}$ fixes $0$, $1$, and $\infty$. Let $w^{\nu}: = \widehat{g} \circ w^{\mu} \circ g^{-1}$. We have 
$$F_g([w^{\mu}]_E) = [w^{\nu}]_{\widetilde{E}} \qquad \mbox { where } \mu = (\nu \circ g)\frac{\bar{g'}}{g'}.$$ 
Therefore, $g_*(\mu | E) = \nu| \widetilde{E}$ and  $g_*: M(E) \to M(\widetilde{E}).$

\smallskip

Next, recall from \S3.5 that $g(S_i) = \widetilde{S}_{\alpha(i)}$ for every $i \in \mathbb{N}$. Let $\mu | S_i = \mu_i$ for each $i \in \mathbb{N}$. By the following diagram 
$$
\begin{array}{ccc}
M(S_i) & {\buildrel (\sigma_g)_i  \over
\longrightarrow} & M(\widetilde{S}_{\alpha(i)})\cr
\downarrow \Phi_i& &\downarrow \widetilde{\Phi}_{\alpha(i)}\cr Teich(S_i) & {\buildrel (\rho_g)_i \over \longrightarrow} &
Teich(\widetilde{S}_{\alpha(i)})
\end{array}
$$
it is clear that $(\rho_g)_i (\Phi_i(\mu_i)) = \widetilde\Phi_{\alpha(i)}(\nu_{\alpha(i)})$ for each $i \in \mathbb{N}$. 

\smallskip

Now, 
$$\mathcal{\widetilde{L}}\big(F_g([w^{\mu}]_E\big) = \mathcal{\widetilde{L}}([w^{\nu}]_{\widetilde{E}}) = \big(\widetilde\Phi(\nu | \widetilde{S}), \nu| \widetilde{E}\big)= \big(\{\widetilde{\Phi}_{\alpha(i)}(\nu_{\alpha(i)})\}, \nu| \widetilde{E}\big).$$

Also,
$$(\rho_g \circ \mathcal{L})\big([w^{\mu}]_E\big) = \rho_g\big(\Phi(\mu|S)\big) = \rho_g\big(\{\Phi_i(\mu_i)\}\big) = \{(\rho_g)_i\big(\Phi_i(\mu_i)\big)\} = \{\widetilde\Phi_{\alpha(i)}(\nu_{\alpha(i)})\}.$$

Finally,
$$\Big((\rho_g \times g_*) \circ \mathcal{L}\Big)\big([w^{\mu}]_E\big) = \Big(\{\widetilde\Phi_{\alpha(i)}(\nu_{\alpha(i)})\}, g_*(\mu | E)\Big) = 
\big(\{\widetilde{\Phi}_{\alpha(i)}(\nu_{\alpha(i)})\}, \nu| \widetilde{E}\big).$$

This proves Theorem A. \qed

\subsection{Conformal naturality for $T(E)$}

Let $G$ be a subgroup of PSL$(2,\C)$ and suppose that $E$ is invariant under $G$ (which means, $g(E) = E$ for all $g$ in $G$). We define the following: 
$$\mbox { (i) }T(E)^G = \{\tau \in T(E): F_g(\tau) = \tau\},$$
$$\mbox { (ii) }Teich(S)^G = \{x \in Teich(S): \rho_g(x) = x\}, $$ and 
$$\mbox { (iii) }M(E)^G = \{\mu|E: (\mu \circ g)\frac{\bar{g'}}{g'} = \mu\}.$$ We then have the following

\begin{corollary} Lieb isomorphism is conformally natural, which means, 
$$T(E)^G = Teich(S)^G \times M(E)^G.$$ 
\end{corollary}
\begin{proof} The corollary follows easily from Theorem A by taking $E = \widetilde{E}$. \end{proof}

\section{Universal holomorphic motion of a closed set.}

\begin{definition}Let $E \subset \hatC$, and let $X$ be a connected Hausdorff space with basepoint $x_0$. 	
	A {\it motion of $E$ over $X$} is a map $\phi: X \times E \to \hatC$ satisfying
	\begin{itemize}
	\item[(i)] $\phi(x_0,z) = z$ for all $z \in E$, and 
        \item[(ii)] for all $x \in X$, the map $\phi(x,\cdot): E \to \hatC$ is injective.
        \end{itemize}
        
        \end{definition}
        
        We say that $X$ is the {\it parameter space} of the motion $\phi$. 
        
        \smallskip

Define $\phi_{x} (\cdot)=\phi(x, \cdot)$ and $\phi^{z} (\cdot)=\phi (\cdot, z)$.
        
        \medskip
        
        We will always assume that $0$, $1$, and $\infty$  belong to $E$ and that the motion $\phi$ is normalized; i.e. $0$, $1$, and $\infty$ are fixed points of the map $\phi_{x}(\cdot)$ for every $x$ in $X$. This can always be achieved after a suitable change of coordinates: choose three distinct points $a$, $b$, $c$ in $E$. For each $x \in X$, let $M_x$ be the M\"obius transformation that maps $\phi_x(a)$, $\phi_x(b)$, $\phi_x(c)$ to $0$, $1$, $\infty$ respectively. Let $\widehat{E} =  M_{x_0}(E)$.  Then the map 
        $\widehat\phi: X \times \widehat{E} \to \hatC$ defined by
  $$\widehat\phi(x, M_{x_0}(z)) = M_x(\phi(x,z))$$
  is a normalized motion.

\medskip

If $E$ is a proper subset of $\widetilde{E}$ and $\phi: X \times E \to \hatC$ and $\widetilde\phi: X \times \widetilde{E} \to \hatC$ are two motions, we say that 
$\widetilde\phi$ {\it extends} $\phi$ if $\widetilde\phi(x,z) = \phi(x,z)$ for all $(x,z) \in X \times E$. 

\medskip

\begin{definition} Let $V$ be a connected complex manifold with basepoint $x_0$. A {\it holomorphic motion of $E$ over $V$} is a motion $\phi: V \times E \to \hatC$  of $E$ over $V$ such that the map
		$\phi^{z}(\cdot): V \to \hatC$ is holomorphic for each $z$ in $E$.
		
		\end{definition}
		
\medskip

As mentioned in the Introduction, holomorphic motions were first studied over the parameter space $\Delta$, motivated by the study of dynamics of rational maps. Let $V = \Delta$. We can think of $\Delta$ as the space of ``time-parameters." 

\medskip

{\bf Example.} Fix $\mu \in L^{\infty}(\C)$ such that $\Vert\mu\Vert_{\infty} = 1$. For each $t$ in $\Delta$, let $\phi(t,z) = w^{t\mu}(z)$ for all $z \in \hatC$, where $w^{t\mu}$ is the (normalized) quasiconformal self-map of $\hatC$ fixing the points $0$, $1$, and $\infty$. Then, by \cite{AB}, the map
$$\phi(t, z) = w^{t\mu}(z)$$
is a holomorphic motion of $\hatC$ over $\Delta$. See also \S1.3 of the Chapter ``A Supplement to Ahlfors's Lectures" by Earle and Kra in \cite{Ah}.

\medskip

The classical $\lambda$-lemma of Ma\~n\'e, Sad, and Sullivan in \cite{MSS} states three surprising properties of holomorphic motions.

\medskip

{\bf The $\lambda$-lemma.} Let $\phi: \Delta \times E \to \hatC$ be a holomorphic motion. Then
\begin{itemize}
	\item[(i)] $\phi$ is a continuous map from $\Delta \times E$ to $\hatC$,
        \item[(ii)] the map $z \mapsto \phi(\lambda, z)$ is the restriction to $E$ of a quasiconformal mapping of the sphere.
        \item[(iii)] $\phi$ extends to a holomorphic motion $\widetilde\phi$ of the closure $\overline{E}$ of $E$.
        \end{itemize}

\medskip

{\bf Slodkowski's extension theorem.} Let $\phi:\Delta \times E \to \hatC$ be a holomorphic motion. Then, there exists a holomorphic motion 
$\widehat\phi: \Delta \times \hatC \to \hatC$, such that $\widehat\phi$ extends $\phi$.

\smallskip

See the papers \cite{AM}, \cite{Chirka}, \cite{GJW}, \cite{Sl}, and the book \cite{Hubbard} for complete proofs. The paper \cite{JMW} gives an alternative proof, using {\T} spaces. 

\medskip

Slodkowski's theorem does not, in general, hold for higher-dimensional parameter spaces. By definition, a holomorphic motion $\phi: B \times E \to \hatC$ is called {\it maximal} if there is no holomorphic motion $\widehat\phi: B \times \widehat E \to \hatC$ such that $E$ is a proper subset of $\widehat E$ and $\widehat\phi$ extends $\phi$. The following example was given in \cite{EM}. 

\medskip

{\bf Example.} Let $B$ be the domain
$$B = \{(\alpha, \beta) \in \C^2: \vert e^{i \alpha}\vert + \vert \beta\vert < 1\}$$
with the basepoint $(i,0)$, and let 
$$E = \{z \in \hatC: \vert z\vert \geq 1\} \cup \{z \in \C: \vert z\vert \leq e^{-1}\}.$$
Then, 
$\phi((\alpha, \beta), z) = z$ if $(\alpha, \beta) \in B, z \in E$, and $\vert z\vert \geq 1$, and 

\smallskip

$\phi((\alpha, \beta), z) = e^{i \alpha}ez + \beta$, if $(\alpha, \beta) \in B, z \in E$, and $\vert z\vert \leq e^{-1}$

\smallskip

defines a maximal holomorphic motion $\phi: B \times E \to \hatC$.

\smallskip

See Theorem 8.1 in \cite{EM}. 

\begin{theorem} Let $E_0$ be any subset of $\hatC$, not necessarily closed; as usual, we assume that $0$, $1$, and $\infty$ belong to $E_0$. Let $E$ be the closure of $E_0$. Let $\phi: V \times E_0 \to \hatC$ be a holomorphic motion where $V$ is a connected complex Banach manifold with a basepoint. There exists a holomorphic motion $\widehat\phi: V \times E \to \hatC$ that extends $\phi$. \end{theorem}

\smallskip

See Theorem 2 in \cite{JM1}. It is a generalized $\lambda$-lemma for any connected complex Banach manifold.

\medskip

\begin{remark} Henceforth, in this paper, we will always assume that $E$ is a closed set in $\hatC$, and as usual, $0$, $1$, and $\infty$ are in $E$, and are fixed points of $\phi_x(\cdot)$ for each $x \in V$.
\end{remark}

\subsection{Universal holomorphic motion of $E$.} The {\it universal holomorphic motion $\Psi_E$ of $E$ over $T(E)$} is defined as follows:
$$\Psi_E(P_E(\mu), z) = w^{\mu}(z) \mbox { for } \mu \in M(\C) \mbox { and } z \in E.$$ 
The definition of $P_E$  above guarantees that $\Psi_E$ is well-defined. It is a holomorphic motion since $P_E$ is a holomorphic split submersion and $\mu \mapsto w^{\mu}(z)$ is a holomorphic map from $M(\C)$ to $\hatC$ for every fixed $z$ in $\hatC$ (by Theorem 11 in \cite{AB}). This holomorphic motion is ``universal" in the following sense:

\medskip

\begin{theorem}~\label{hm}
Let $\phi: V \times E \to \hatC$ be a holomorphic motion where $V$ is a simply connected complex Banach manifold with a basepoint $x_0$. Then, there exists a unique basepoint preserving holomorphic map $f: V \to T(E)$ such that $f^*(\Psi_E) = \phi$; which means, $\phi(x,z) = \Psi_E(f(x), z)$ for all $(x,z) \in V \times E$. 
\end{theorem}

For a proof, see \S14 in \cite{M1}.

\subsection{The case when $E = \hatC$}
When $E = \hatC$, the space $T(\hatC)$ consists of all the
normalized quasiconformal self-mappings of $\hatC$, and the map
$P_{\hatC}$ from $M(\C)$ to $T(\hatC)$ is bijective. We use it to
identify $T(\hatC)$ biholomorphically with $M(\C)$. The pullback
$\widetilde\Psi_{\hatC}$ of $\Psi_{\hatC}$ to $M(\C)$ by
$P_{\hatC}$ satisfies
$$\widetilde\Psi_{\hatC}(\mu,z) = \Psi_{\hatC}(P_{\hatC}(\mu),z) =
  w^{\mu}(z), \qquad (\mu,z) \in M(\C) \times \hatC.$$

Thus, the universal holomorphic motion of $\hatC$ becomes the map
\begin{equation}\label{eq:3.1}
  \Psi_{\hatC}(\mu,z) = w^{\mu}(z)
\end{equation}
from $M(\C)\times\hatC$ to $\hatC$.

\smallskip

In particular, if $\phi: V \times \hatC \to \hatC$ is a
holomorphic motion and $V$ is a simply connected complex Banach manifold, Theorem 2
provides a unique basepoint preserving holomorphic map $f:V \to
M(\C)$ such that
$$\phi(x,z) = \Psi_{\hatC}(f(x),z) = w^{f(x)}(z)$$
for all $x$ in $V$ and $z$ in $\hatC$. In other words, $f(x)$ is
the Beltrami coefficient of the quasiconformal map $\phi_x(z)$ for
each $x$ in $V$.

\begin{remark} The map $f: V \to M(\C)$ that sends $x$ in $V$ to the
Beltrami coefficient of $\phi_x(z)$ is well defined for every
holomorphic motion $\phi: V \times \hatC \to \hatC$. By Theorem
2, it is holomorphic when $V$ is a simply connected complex Banach manifold. It follows
readily that $f$ is holomorphic for every connected complex
manifold $V$. See Theorem 4 in \cite{Earle}. For $V = \Delta$, see
Theorem 2 in \cite{BR}.
\end{remark}

\section{Douady-Earle section}

\subsection{Background}

We give a brief background of the {\it Douady-Earle section} (sometimes called the {\it barycentric section}) for classical {\T} spaces; see \cite{DE} for details. 

\smallskip

Let $M(\Delta)$ denote the open unit ball of the complex Banach space $L^{\infty}(\Delta,\C)$. For each $\mu \in M(\Delta)$, there exists a unique quasiconformal map $f^{\mu}$ of $\overline\Delta$ onto itself fixing the points $1$, $i$, and $-1$. Let $\varphi^{\mu}$ be the restriction of $f^{\mu}$ to the unit circle $S^{1}$. Let $ex(\varphi^{\mu}): \overline\Delta \to \overline\Delta$ denote the Douady-Earle extension of $\varphi^{\mu}$. By Theorem 2 in \cite{DE}, $ex(\varphi^{\mu})$ is quasiconformal and so its complex dilatation belongs to $M(\Delta)$. That determines a map $\sigma: M(\Delta) \to M(\Delta)$ that sends $\mu$ to the Beltrami coefficient of $ex(\varphi^{\mu})$; more precisely, we have:
\begin{equation}
\label{eq:DE}
\sigma: \mu \mapsto ex(\varphi^{\mu})_{\overline z}/ex(\varphi^{\mu})_{z}
\end{equation}
from $M(\Delta)$ to $M(\Delta)$. 
 See \S6 of \cite{DE} for details. In \S6 of \cite{DE} it is shown that the map $\sigma$ is conformally natural. 

\smallskip

Let $G$ be the group of all conformal automorphisms of $\Delta$. Let $\Gamma$ be a Fuchsian group (discrete subgroup of $G$). 
Let $M(\Gamma)$ be the $\Gamma$-invariant elements of $M(\Delta)$, $Teich(\Gamma)$ be the {\T} space of $\Gamma$ and $\pi: M(\Gamma) \to Teich(\Gamma)$ be the usual projection. The following definitions are standard; see \S7 of \cite{DE} for details.
$$M(\Gamma) = \{\mu \in M(\Gamma): \mu = (\mu \circ g)\frac{{\overline g}'}{g'}\}.$$
The {\T} space $Teich(\Gamma)$ is defined by
$$Teich(\Gamma) = \{\varphi \in \mathcal H(S^1); \varphi = \varphi^{\mu} \mbox { for some } \mu \in M(\Gamma)\}.$$
Here, $\mathcal H(S^1)$ denotes the set of all homeomorphisms $\varphi: S^1 \to S^1$. Let $\pi: M(\Gamma) \to Teich(\Gamma)$ be given by $\pi(\mu) = \varphi^{\mu}$.

\smallskip

In Lemma 5 of \cite{DE} it is shown that:

\begin{itemize}

\item[(i)] $\sigma$ maps $M(\Gamma)$ into itself,

\item[(ii)] there is a continuous map $s: Teich(\Gamma) \to M(\Gamma)$ such that $s \circ \pi = \sigma: M(\Gamma) \to M(\Gamma)$, and 

\item[(iii)] $\pi \circ \sigma = \pi: M(\Gamma) \to Teich(\Gamma)$.

\end{itemize}

We call the continuous map $s: Teich(\Gamma) \to M(\Gamma)$ the {\it Douady-Earle section} of $\pi$ for the {\T} space $Teich(\Gamma)$. 

\subsection{Real-analyticity of Douady-Earle section for Teichm\"uller spaces}

Theorem 4 in \cite{DE} shows the extremely important property that the map $\sigma$ is real-analytic. Since $\pi: M(\Gamma) \to Teich(\Gamma)$ is a holomorphic split submersion (which means that $\pi$ has local holomorphic sections), and $\sigma$ is real-analytic, it clearly follows from the discussion in \S5.1 that the map $s: Teich(\Gamma) \to M(\Gamma)$ is real-analytic. 

\begin{remark}
The term {\it Douady-Earle section} was first used in the paper \cite{JM2}. The real-analyticity of the Douady-Earle section is not explicitly stated in the paper \cite{DE}. For universal {\T} space, the real-analyticity was proved in \cite{Earle-Marden}; see Proposition 3.9 of that paper (the authors call it ``barycentric section"). For the reader's convenience, and to make our paper self-contained, we include a discussion for classical {\T} spaces. More details will appear in a forthcoming paper of the third author with Professor Yunping Jiang. 
\end{remark}

\medskip

Let $X$ be a hyperbolic Riemann surface, and let $X = \Delta/\Gamma$ where $\Gamma$ is a uniformizing Fuchsian group of $X$. For any $g$ in Aut$(\Delta)$, set $\Gamma_g = g^{-1}\Gamma g$. Let $g_{*}: M(\Gamma) \to M(\Gamma_g)$ be the induced map defined by
$$g_{*}(\mu) = (\mu \circ g)\frac{\overline{g'}}{g'} \qquad \mbox { for } \mu \in M(\Gamma).$$
It is well-known that $g_{*}$ is biholomorphic and that $g_{*}$ induces a biholomorphic map $\widehat g: Teich(\Gamma) \to Teich(\Gamma_g)$. 

\medskip

As usual, let $M(X)$ denote the open unit ball of Belt$(X)$, the complex Banach space of Beltrami forms on $X$. Let $Teich(X)$ denote
the {\T} space of $X$ and let $\Phi: M(X) \to Teich(X)$ be the standard projection. It is well-known that $\Phi$ is a holomorphic split submersion.
It is standard in {\T} theory that there are canonical biholomorphic maps $\rho: M(X) \to M(\Gamma)$ and $\rho_{*}: Teich(X) \to Teich(\Gamma)$; see, for example, Sections 2.1.4, and 2.2.6 in \cite{Nag} for all details, and, in particular, \S\S3.3, 4, 7 in \cite{E}. 
The paper \cite{E}, based on the author's lectures at Warwick in 1992, and then at the Harish-Chandra Research Institute, India, in January 2025, contains a detailed discussion.

\medskip

We have the standard projections $$\pi: M(\Gamma) \to Teich(\Gamma)$$ and 
$$\pi_g: M(\Gamma_g) \to Teich(\Gamma_g)$$
(which are holomorphic split submersions); let $s_{\pi}$ denote the Douady-Earle section of $\pi$ and let $s_{\pi_g}$ denote the Douady-Earle section of $\pi_g$. By our discussion in \S5.2, both maps $s_{\pi}$ and $s_{\pi_g}$ are real-analytic. 

\medskip

The Douady-Earle section $S: Teich(X) \to M(X)$ is defined as follows: $S = \rho^{-1} \circ s_{\pi} \circ \rho_{*}$. Note that $\Phi \circ S$ is the identity on $Teich(X)$.

\medskip

The following proposition shows that the Douady-Earle section $S: Teich(X) \to M(X)$ is well-defined, which means, it is independent of the covering transformations. 

\begin{prop}
$g_{*} \circ s_{\pi} = s_{\pi_g} \circ \widehat g$. 
\end{prop}

\begin{proof}
We have the following diagram.

\[\begin{tikzcd}
	{M(X)} && {M(\Gamma)} && {M(\Gamma_g)} \\
	\\
	{Teich(X)} && {Teich(\Gamma)} && {Teich(\Gamma_g)}
	\arrow["\rho", from=1-1, to=1-3]
	\arrow["\Phi"', shift right=3, from=1-1, to=3-1]
	\arrow["{g_*}", from=1-3, to=1-5]
	\arrow["\pi"', shift right=3, from=1-3, to=3-3]
	\arrow["{\pi_g}"', shift right=3, from=1-5, to=3-5]
	\arrow["S", shift right=3, from=3-1, to=1-1]
	\arrow["{\rho_*}"', from=3-1, to=3-3]
	\arrow["{s_{\pi}}"', shift right=3, from=3-3, to=1-3]
	\arrow["{\widehat g}"', from=3-3, to=3-5]
	\arrow["{s_{\pi_g}}"', shift right=3, from=3-5, to=1-5]
\end{tikzcd}\]

Let $\tau \in Teich(\Gamma)$, where $\tau = \pi(\mu)$ for $\mu \in M(\Gamma)$. Also, let $g_{*}(\mu) = \widetilde\mu \in M(\Gamma_g)$. 
We have 
$$g_{*}(s_{\pi}(\tau)) = g_{*}(s_{\pi}(\pi(\mu))) = g_{*}(\sigma(\mu)).$$
We also have 
$$s_{\pi_g}(\widehat g(\tau)) = s_{\pi_g}(\widehat g(\pi (\mu))) = s_{\pi_g}(\pi_g(g_*(\mu)) = s_{\pi_g}(\pi_g(\widetilde\mu)) = \sigma(\widetilde\mu).$$
Finally, because of conformal naturality of the map $\sigma$, we have
$$g_{*}(\sigma(\mu)) = (\sigma(\mu) \circ g)\frac{\overline g'}{g'} = \sigma ((\mu \circ g)\frac{\overline g'}{g'}) = \sigma(\widetilde\mu).$$
Hence, we conclude that 
$$g_{*}(s_{\pi}(\tau)) = s_{\pi_g}(\widehat g(\tau)) \qquad \mbox { for } \tau \in Teich(\Gamma).$$
\end{proof}

\subsection{Douady-Earle section for product Teichm\"uller spaces}

We now study Douady-Earle section for the product {\T} space $Teich(E^c)$ defined in \S3.4.

\begin{prop} There is a continuous basepoint preserving map $\widehat{s}$ from $Teich(E^c)$ to $M(E^c)$ such that $\Phi \circ \widehat{s}$ is the identity map on $Teich(E^c)$. \end{prop}

\begin{proof} Let $\tau \in Teich(E^c)$ where $\tau = \Phi(\mu)$ for $\mu \in M(E^c)$. Following our discussion in \S3.4, $\tau = \{\tau_n\}$ where $\tau_n = \Phi_n(\mu_n) \in Teich(S_n)$ where $\mu_n \in M(S_n)$ for each $n \geq 1$. By Lemma 5 in \cite{DE} (also the above discussion), for each $n \geq 1$, there is a continuous basepoint preserving map $\widehat{s_n}$ from $Teich(S_n)$ into $M(S_n)$ such that $\Phi_n \circ \widehat{s_n}$ is the identity map on $Teich(S_n)$. Let $\sigma_n$ denote the continuous map $\widehat{s_n} \circ \Phi_n$ from $M(S_n)$ to itself.

Since $\tau \in Teich(E^c)$, we have by (3.1), $\sup\{d_n(0_n,\tau_n): n \geq 1\} < \infty$. Let $\Vert\mu_n\Vert_{\infty} \leq k$ for all $n \geq 1$. Then, by Proposition 7 in \cite{DE}, there exists $0 \leq c(k) < 1$, where $c(k)$ depends only on $k$ and is independent of $n$, such that
$$\Vert\sigma_n(\mu_n)\Vert_{\infty} \leq c(k) \mbox { for all } n \geq 1.$$
Define $\sigma(\mu): = \{\sigma_n(\mu_n)\}$. It is easy to check that $\sigma$ is a continuous map of $M(E^c)$ into itself (see, for example, Proposition 7.11 in \cite{EM}). Furthermore, there exists a unique well defined map $\widehat{s}$ from $Teich(E^c)$ to $M(E^c)$ such that $\sigma = \widehat{s} \circ \Phi$. Since $\sigma$ is continuous and $\Phi$ is a holomorphic split submersion, it follows that $\widehat{s}$ is continuous. It is easy to check that $\Phi \circ \widehat{s}$ is the identity map on $Teich(E^c)$.
\end{proof}

\begin{definition} The map $\widehat{s}$ from $Teich(E^c)$ to $M(E^c)$ is called the {\it Douady-Earle section} of $\Phi$ for the product {\T} space $Teich(E^c)$. \end{definition}

\subsection{Douady-Earle section for $T(E)$}

Finally, we introduce the Douady-Earle section for the {\T} space of the closed set $T(E)$ defined in \S\S 3.1, 3.2.

\begin{prop}  There is a continuous basepoint preserving map $s$ from $T(E)$ to $M(\C)$ such that $P_E \circ s$ is the identity map on $T(E)$.  \end{prop}

\begin{proof} By Proposition 3, there is a continuous basepoint preserving map $\widehat{s}$ from $Teich(E^c)$ to $M(E^c)$ such that $\Phi \circ \widehat{s}$ is the identity map on $Teich(E^c)$. Let $\widetilde{s}$ be the map from $Teich(E^c) \times M(E)$ to $M(\C)$ such that $\widetilde{s}(\tau, \nu)$ equals $\widehat{s}(\tau)$ in $E^c$ and equals $\nu$ in $E$ for each $(\tau, \nu)$ in $Teich(E^c) \times M(E)$. Clearly, $\widetilde{P}_E \circ \widetilde{s}$ is the identity map on $Teich(E^c) \times M(E)$.
We define $s = \widetilde{s} \circ \mathcal L$, where $\mathcal L$ is the biholomorphic map from $T(E)$ to $Teich(E^c) \times M(E)$ given in Definition 3. It follows that $s: T(E) \to M(\C)$ is a continuous basepoint preserving map such that $P_E \circ s$ is the identity map on $T(E)$. \end{proof}

\begin{definition} The map $s$ from $T(E)$ to $M(\C)$ is called the {\it Douady-Earle section} of $P_E$ for the {\T} space $T(E)$.
\end{definition}

Since $M(\C)$ is contractible, we have the following

\begin{corollary} The Teichm\"uller space $T(E)$ is contractible. \end{corollary}

\begin{remark} Let $t \in T(E)$ and $P_E(\mu) = t$ for $\mu \in M(\C)$. By Proposition 1, we have
$$\mathcal L(t) = \mathcal L(P_E(\mu)) = \widetilde{P_E}(\mu) = \Big(\Phi(\mu|E^c), \mu|E\Big).$$ Let $\Phi(\mu|E^c)$ be denoted by $\tau$. By Proposition 3, $s(t) = \widetilde{s}(\mathcal L(t))$, which equals $\widehat{s}(\tau)$ on $E^c$, and equals $\mu$ on $E$. By Proposition 3, we have
$\widehat{s}(\tau) = \widehat{s}(\Phi(\mu|E^c)) = \sigma(\mu) \mbox { on } E^c$. Thus, for $t  = P_E(\mu)$ in $T(E)$, $s(t)$ equals $\sigma(\mu)$ on $E^c$ and equals $\mu$ on $E$. If $\Vert\mu\Vert_{\infty} = k$, then $\Vert s(t)\Vert_{\infty} \leq$ max$(k, c(k))$ where $c(k)$ depends only on $k$, and $0 \leq c(k) < 1$.
\end{remark}

\begin{corollary} For $t$ in $T(E)$, $\Vert s(t)\Vert_{\infty}$ is bounded above by a number between $0$ and $1$, that depends only on $d_{T(E)}(0,t)$.  \end{corollary}

\begin{proof} Given $t$ in $T(E)$, choose an extremal $\mu$ in $M(\C)$ so that $P_E(\mu) = t$. Then
$$d_{T(E)}(0,t) = \frac{1}{2}\log K \mbox { where } K = \frac{1+k}{1-k} \mbox { and } k = \Vert\mu\Vert_{\infty}.$$
By Remark 4, we have $\Vert s(t)\Vert_{\infty} \leq$ max$(c(k), k)$. \end{proof}

\subsection{Conformal naturality of the Douady-Earle section for $T(E)$} 

For this paper, we do not need the conformal naturality of the Douady-Earle section for $T(E)$; but it is worth including a brief discussion. See \S3.4 in \cite{JM2} for the details.

As in \S3.7, let $G$ be a group of M\"obius transformations that map $E$ onto
itself. For each $g$ in $G$, there exists a biholomorphic map
$F_g: T(E) \to T(E)$ which is defined as follows: for each $\mu$ in
$M(\C)$,
\begin{equation}
\label{eq:egl}
F_g([w^{\mu}]_E) = [\widehat g \circ w^{\mu} \circ
g^{-1}]_E
\end{equation}
where $\widehat g$ is the unique M\"obius transformation such that
$\widehat g \circ w^{\mu} \circ g^{-1}$ fixes the points 0, 1, and
$\infty$. See Remark 3.4 in \cite{EGL} for a discussion on
``geometric isomorphisms" of $T(E)$.

\smallskip

It follows from the definition that, for each $g$ in $G$, $F_g$ is
basepoint preserving.

\begin{definition}
We define $M(\C)^G$ and $T(E)^G$ as follows:
$$M(\C)^G:= \{\mu \in M(\C): (\mu \circ g)\frac{\bar{g'}}{g'} = \mu \mbox { a.e.} \mbox { on } \C \mbox { for each } g \in G\}$$ and
$$T(E)^G: = \{t \in T(E): F_g(t) = t \mbox { for each } g \in G\}.$$
\end{definition}

The next proposition shows the conformal naturality of the Douady-Earle section $s: T(E) \to M(\C)$. As we will see, it has important applications in holomorphic motions.

\begin{prop} If $t \in T(E)^G$, then $s(t) \in M(\C)^G$. \end{prop}

See Proposition 3.9 in \cite{JM2}.

\medskip

{\bf Open question.} Is the Douady-Earle section $s: T(E) \to M(\C)$ real-analytic? An affirmative answer to this question, together with the conformal naturality of $s$ will have several applications in holomorphic motions and some questions in geometric function theory. The third author of this paper and Professor Yunping Jiang are presently studying this question.

\section{A maximal holomorphic motion}

We begin with a general lemma.  Suppose $E_1$ and $E_2$ are
closed subsets of $\hatC$ such that $E_1$ is a proper subset of
$E_2$, and as usual, $0, 1, \infty$ belong to $E_1$. If $\mu$ is
in $M(\C)$, then the $E_2$-equivalence class of $w^{\mu}$ is
contained in the $E_1$-equivalence class of $w^{\mu}$. Therefore,
there is a well-defined `forgetful map' $p_{E_2,E_1}$ from
$T(E_2)$ to $T(E_1)$ such that $P_{E_1} = p_{E_2,E_1}\circ
P_{E_2}$. It is easy to see that this forgetful map is a basepoint
preserving holomorphic split submersion. We also have the
universal holomorphic motions $\Psi_{E_1}: T(E_1) \times E_1 \to
\hatC$ and $\Psi_{E_2}: T(E_2) \times E_2 \to \hatC$. We need the
following lemma.

\begin{lemma} Let $V$ be a connected complex Banach manifold with
basepoint $x_0$ and let $f$ and $g$ be basepoint preserving
holomorphic maps from $V$ into $T(E_1)$ and $T(E_2)$,
respectively. Then $p_{E_2,E_1}\circ g = f$ if and only if
$g^*(\Psi_{E_2})$ extends $f^*(\Psi_{E_1})$. \end{lemma}

See Proposition 13.1 in \cite{M1} for a complete proof.

For a finite set $E$, its complement $\hatC \setminus E = \Omega$
is a Riemann sphere with punctures at the points of $E$. There is
a natural identification of $T(E)$ with the classical
Teichm\"uller space $Teich(\Omega)$ (i.e. the Teichm\"uller space
of the sphere with punctures at the points of $E$). See \cite{M1} for details. 
This also follows easily from {\it Lieb isomorphism}. 

\bigskip

{\bf Theorem B.} Let $E$ be the finite set $\{0, 1, \infty,
\zeta_1, \cdot \cdot \cdot, \zeta_n\}$, where $\zeta_i \not=
\zeta_j$ for $i \not= j$ and $n \geq 2$. Consider the universal
holomorphic motion $\Psi_E: T(E) \times E \to \hatC$. Then, the following are true:

\smallskip

(i) The map $\widetilde\Psi_E: T(E) \times \hatC \to
\hatC$ defined by the formula
$$\widetilde\Psi_E(t,z) = w^{s(t)}(z), \qquad (t,z) \in T(E) \times
\hatC$$ 
extends the holomorphic motion $\Psi_E$. 

\smallskip

(ii) The map $\widetilde\Psi_E$ is continuous. 

\smallskip

(iii) The Beltrami coefficient $s(t)$ varies real-analytically on $T(E)$.

\smallskip

(iv) The universal
holomorphic motion $\Psi_E: T(E) \times E \to \hatC$ is a maximal
holomorphic motion. 

\medskip

{\bf Proof of Theorem B.} 

First note that
$$\Psi_E(t,z) = \Psi_E(P_E(s(t)), z) = w^{s(t)}(z) =
\widetilde\Psi_E(t,z)$$ for all $(t,z) \in T(E) \times E$.
Therefore, $\widetilde\Psi_E$ extends $\Psi_E$.

\smallskip

The continuity of $\widetilde\Psi_{E}$ follows from Lemma 17 of
\cite{AB}, which says that $w^{\mu_n} \to w^{\mu}$ uniformly in
the spherical metric if $\mu_n \to \mu$ in $M(\C)$. 

\smallskip

By the construction of $\widetilde\Psi_E$, for each $t \in T(E)$, the map 
$(\widetilde\Psi_E)_t: \hatC \to \hatC:= w^{s(t)}:\hatC \to \hatC$ is clearly a quasiconformal map, with Beltrami coefficient $s(t)$, 
and by our discussion in \S5.2, $s(t)$ depends real-analytically on $t$, Note also that, by Corollary 3, the $L^{\infty}$ norm of $s(t)$ depends only on $d_{T(E)}(0,t)$. 

\smallskip

Finally, let $\widehat{E}
= \{0, 1, \infty, \zeta_1, \cdot \cdot \cdot, \zeta_n,
\zeta_{n+1}\}$ where $\zeta_{n+1}$ is any point in $\hatC
\setminus \{0,1, \infty\}$ different from all the points in the
set $E$.  (Using
classical notations, we denote by $Teich(g,n)$ the Teichm\"uller
space of a Riemann surface of genus $g$ and with $n$ punctures.
So, in our example, $T(E)$ is $Teich(0, n+3)$ and $T(\widehat{E})$
is $Teich(0, n+4)$.)

\smallskip

Note that in this case $p_{\widehat{E},E}$ is the usual
puncture-forgetting map in Teichm\"uller theory, which is known to
be a basepoint preserving holomorphic split submersion (see
\cite{Nag}).

\smallskip

Consider the universal holomorphic motion
$$\Psi_E: T(E) \times E \to \hatC.$$
Let $i: T(E) \to T(E)$ be the identity map, which is obviously a
basepoint preserving holomorphic map. Let $\phi: T(E) \times E \to
\hatC$ be the holomorphic motion $i^*(\Psi_E)$ (which is the same
as $\Psi_E$). Suppose $\phi$ extends to a holomorphic motion
$\widehat\phi: T(E) \times \widehat{E} \to \hatC$. Then, since
$T(E)$ is simply connected (see Corollary 2), it follows by
Theorem 2 that there exists a unique basepoint preserving
holomorphic map $f: T(E) \to T(\widehat{E})$ such that
$f^*(\Psi_{\widehat{E}}) = \widehat\phi$, where $\Psi_{\widehat
E}: T(\widehat E) \times \widehat E \to \hatC$ is the universal
holomorphic motion of $\widehat E$. Since $\widehat\phi$ extends
$\phi$, it follows by Lemma 1, that $p_{\widehat E, E} \circ f =
i$. That means, the map $p_{\widehat E, E}$ has a holomorphic
section $f$. That contradicts a result of Earle and Kra in
\cite{EK}. Hence, we conclude that the universal holomorphic
motion $\Psi_E: T(E) \times E \to \hatC$ is a maximal holomorphic
motion.\qed

\section{Holomorphic families of Jordan curves}

{\bf Theorem C.} Let $V$ be a simply connected complex Banach manifold with a basepoint $x_0$. Let $\gamma_{x_0}$ be a closed Jordan curve, and let $E$ be a finite subset of $\gamma_{x_0}$, containing the points $0$, $1$, and $\infty$. Let $\phi: V \times E \to \hatC$ be a holomorphic motion and let $E_x: = \phi_x(E)$. Then we have the following:

(i) For each $x$ in $V$, there is a quasiconformal map $\widetilde\phi_x: \hatC \to \hatC$ such that $\widetilde\phi_x(E) = E_x$
and $\widetilde\phi_x(\gamma_{x_0})$ is a closed Jordan curve $\gamma_x$, and $E_x \subset \gamma_x$. 

\smallskip

(ii) The Beltrami coefficient $\mu_x$ of $\widetilde\phi_x$ varies real-analytically on $V$.

\smallskip

(iii)  The $L^{\infty}$ norm of $\mu_x$ is bounded above by a number
less than 1, that depends only on the Kobayashi distance from $x$
to $x_0$, denoted by $\rho_V(x,x_0)$.

\medskip

{\bf Proof of Theorem C.} By Theorem 2, there exists a unique holomorphic map $f: V \to T(E)$ such that $f^*(\Psi_E) = \phi$
where $\Psi_E: T(E) \times E \to \hatC$ is the universal holomorphic motion of $E$ defined in \S4.1. By our discussion in \S5.2, there exists a real-analytic map $s: T(E) \to M(\C)$ such that $P_E \circ s$ is the identity map on $T(E)$. Let $\widetilde f = s \circ f$. Define the map $\widetilde\phi: V \times \hatC \to \hatC$ as follows: 
$$\widetilde\phi(x,z): = \widetilde\phi_x(z) = w^{\widetilde f(x)}(z) \qquad \mbox { for all } z \in \hatC.$$
Then, for each $z$ in $E$, we have
$$\phi_x(z) = \phi (x,z) = \Psi_E(f(x), z) = \Psi_E(P_E(\widetilde f(x)), z) = w^{\widetilde f(x)}(z) = \widetilde\phi_x(z).$$
Hence, for each $x$ in $V$, $\widetilde\phi_x(E) = \phi_x(E) = E_x$. Since for each $x$ in $V$, $\widetilde\phi:\hatC \to \hatC$ is a homeomorphism, $\widetilde\phi(\gamma_{x_0}) = \gamma_x$ is a Jordan curve, and clearly, $E_x \subset \gamma_x$. This proves (i) and (ii).

\smallskip

Finally, let $x$ be in $V$ where $x \not= x_0$. Since the
Teichm\"uller metric on $T(E)$ is the same as its Kobayashi metric
(see \S3.3), we have $d_{T(E)}(0,t) \leq \rho_V(x_0,x)$ where
$f(x) = t$ and $0$ denotes the basepoint in $T(E)$. Choose an
extremal $\mu$ in $M(\C)$ such that $P_E(\mu) = f(x)$. This means
that $d_{T(E)}(0,P_E(\mu)) = d_M(0_M,\mu)$ where $0_M$ denotes the
basepoint in $M(\C)$. We have
$$d_{T(E)}(f(x_0), f(x)) =
\frac{1}{2}\log\frac{1+\Vert\mu\Vert_{\infty}}{1-\Vert\mu\Vert_{\infty}}
\leq \rho_V(x_0,x)$$ which gives
$$\Vert\mu\Vert_{\infty} \leq \frac{\exp(2\rho_V(x_0,x))
-1}{\exp(2\rho_V(x_0,x)) +1} < 1.$$

\smallskip

Since $\widetilde\phi_x(z) = w^{\widetilde f(x)}(z)$, it follows
from Corollary 3 that $\Vert \widetilde{f}(x)\Vert_{\infty}$ is
bounded above by a number between 0 and 1, that depends only on
$\rho_V(x_0,x)$. This proves (iii). \qed

\begin{remark}
The map $\widetilde\phi: V \times \hatC \to \hatC$ is continuous. For, suppose $(x_n,z_n) \to (x,z)$ where $x_n$, $x$ are in $V$ and
$z_n$, $z$ are in $\hatC$. Since $\widetilde{f}: V \to M(\C)$ is real-analytic, and so
continuous, we have $\widetilde{f}(x_n) \to \widetilde{f}(x)$. The
continuity of $\widetilde\phi$ now follows exactly like the proof of (ii) in Theorem B. 

\end{remark}

\subsection{Some historical remarks}

Let $\gamma_{0}$ be a closed Jordan curve in $\hatC$ and let $\phi: \Delta \times \gamma_{0} \to \hatC$ be a holomorphic motion; here, the basepoint of the parameter space $\Delta$ is obviously $0$. For each $\lambda \in \Delta$, by the $\lambda$-lemma, 
$$\gamma_{\lambda} = \{ \phi_{\lambda}(z): z \in \gamma_{0}\}$$ 
is a closed Jordan curve in $\C$. 

\medskip

Consider the family $\{\gamma_{\lambda}: \lambda \in \Delta\}$ generated by a holomorphic motion $\phi: \Delta \times \gamma_0 \to \hatC$. In
\cite{GP}, it was proved that if $\gamma_{0}$ is a quasicircle, then each $\gamma_{\lambda}$ is also a quasicircle. Later, in \cite{PR}, the authors proved the following generalization.

\begin{theorem}[{\bf Pommerenke and Rodin}] Let $\{\gamma_{\lambda}: \lambda \in \Delta\}$ be generated by a holomorphic motion of $\gamma_0$. For each $\lambda \in \Delta$ there exists a quasiconformal homeomorphism of $\hatC$ which maps $\gamma_0$ onto $\gamma_{\lambda}$. 
\end{theorem}

See Theorem 2 in \cite{PR}. 

\medskip

Here is a much stronger version of this theorem of Pommerenke and Rodin.

\begin{theorem}
Let $V$ be a simply connected complex Banach manifold with a basepoint $x_0$. Let $\gamma_{x_0}$ be a closed Jordan curve, and let 
$\phi: V \times \gamma_{x_0} \to \hatC$ be a holomorphic motion. Let $\{\gamma_x: x \in V\}$ be generated by the holomorphic motion $\phi$. 
Then we have the following:

\begin{itemize}

\item[(i)]For each $x$ in $V$, there is a quasiconformal map $\widetilde\phi_x: \hatC \to \hatC$ which maps $\gamma_{x_0}$ onto $\gamma_x$. 

\item[(ii)]The Beltrami coefficient $\mu_x$ of $\widetilde\phi_x$ varies continuously with respect to $x$ in $V$. 

\item[(iii)]The $L^{\infty}$ norm of $\mu_x$ is bounded above by a number less than 1, that depends only on the Kobayashi distance from $x$
to $x_0$, denoted by $\rho_V(x,x_0)$.

\end{itemize}

\end{theorem}

The proof is exactly similar to the proof of Theorem C.

\begin{remark}
An affirmative answer to the open question stated at the end of
§5 will prove that the Beltrami coefficient $\mu_x$ of $\widetilde\phi_x$ varies real-analytically on $V$ . That would be a substantial improvement of Theorem 3, and also of Theorem 4.

\end{remark}

\begin{remark}
When $V = \Delta$, the Beltrami coefficient $\mu_x$ of $\widetilde\phi_x$ varies holomorphically with respect to $x$ in $V$, and $\Vert\mu_x\Vert_{\infty} \leq \vert x\vert$. This is a consequence of Slodkowski's theorem. 

\end{remark}

\end{document}